\documentclass[reqno]{amsart}

\usepackage{amssymb}
\usepackage{cite}

\numberwithin{equation}{section}
\newtheorem{Theorem}{Theorem}[section]

\newtheorem{Corollary}[Theorem]{Corollary}
\newtheorem{Lemma}[Theorem]{Lemma}
\newtheorem{Proposition}[Theorem]{Proposition}
\newtheorem{Remark}{Remark}[section]

\title[Fractional Chain rule in Sobolev spaces]%
{Remark on the Chain rule of fractional derivative in the Sobolev framework}
\author[K. Fujiwara]{Kazumasa Fujiwara}

\address{%
Graduate School of Mathematics, Nagoya University,
Furocho, Chikusaku, Nagoya Japan, 464-8602 }

\email{fujiwara.kazumasa@math.nagoya-u.ac.jp}

\begin{document}

\maketitle

\begin{abstract}
A chain rule for power product is studied
with fractional differential operators
in the framework of Sobolev spaces.
The fractional differential operators are defined by the Fourier multipliers.
The chain rule is considered newly in the case where the order of differential operators
is between one and two.
The study is based on the analogy of the classical chain rule or Leibniz rule.
\end{abstract}

\section{Introduction}
The chain rule or Leibniz rule is an essential tool
to study nonlinear differential equations.
In the study of nonlinear partial differential equations(PDEs),
fractional differential operators are also known as powerful tools.
So, in order to analyze nonlinear PDEs,
chain rules for fractional differential operators are naturally required.
Even though fractional differential operators may be non-local unlike classical operators,
estimates for fractional derivative has been studied
on the analogy of classical chain rules.
The history of this study can go back at least to the work of Kato and Ponce \cite{KP88}.

In this paper,
we consider a chain rule corresponding to the identity:
	\[
	\frac{d}{dx} F(u) = F'(u) u',
	\]
in the framework of Riesz potential space $\dot H_p^s = D^{-s} L^p(\mathbb R^d)$,
where $s \in \mathbb R$, $d \geq 1$, and $L^p$ is the usual Lebesgue space.
$\dot H_p^s$ is also called as homogeneous Sobolev space.
The fractional differential operator $D^s = (-\Delta)^{s/2}$ is recognized as a Fourier multiplier by
$D^{s} = \mathfrak F^{-1} |\cdot|^s \mathfrak F$,
where $\mathfrak F$ is the standard Fourier transform.
Especially, we study the case where $F$ behaves like power product, that is,
$F(z) \sim |z|^{p-1} z$.

In \cite{CW91}, Christ and Weinstein showed the following estimate:
\begin{Proposition}[{\cite[Proposition:3.1]{CW91}}]
\label{Proposition:1.1}
Let $d \geq 1$ and $s \in (0,1)$.
Let $F \in C(\mathbb C)$ and $G \in C(\mathbb C:[0,\infty))$ satisfy
	\[
	|F(u) - F(v)|
	\leq (G(u)+G(v))|u-v|.
	\]
Let $1 \leq p < \infty$, $1 < r, q < \infty$ satisfy
	\begin{align}
	\frac 1 p = \frac 1 q + \frac 1 r.
	\label{eq:1.1}
	\end{align}
The estimate
	\[
	\| F(u) \|_{\dot H_p^s}
	\leq \| G(u) \|_{L^q} \| u \|_{\dot H_r^s}.
	\]
holds for any $u \in \dot H_r^s$ with $G(u) \in L^q$.
\end{Proposition}
Roughly speaking, Proposition \ref{Proposition:1.1} asserts that
$D^s F(u)$ behaves like $F'(u) D^s u$.
Since the Riesz operator $R = D^{-1} \nabla$ is bounded on $L^p$ when $p \in (1,\infty)$,
$\| F(u) \|_{\dot H_p^s}$ may be estimated even when $s \geq 1$
by combining the classical chain rule, Proposition \ref{Proposition:1.1}, and the H\"older inequality.

On the other hand, when $F(z) = |z|^{\rho-1} z$,
Proposition \ref{Proposition:1.1} may be insufficient to handle general cases.
For example, when $\rho \in (1,2)$ and $s \in (1,\rho)$,
one may regard
	\[
	D^s F(u) \sim D^{s-1} (|u|^{\rho-1} \nabla u).
	\]
By using the fractional Leibniz rule
(for example, see \cite[Theorem 1]{GO14}, \cite{ES20}, \cite{FGO18}, and references therein),
one can distribute $D^{s-1}$ like $\nabla$ to $\nabla u$ and $|u|^{\rho-1}$.
Since $f(z) = |z|^{\rho-1}$ is not differentiable,
one cannot control $D^{s-1} |u|^{\rho-1}$ by applying Proposition \ref{Proposition:1.1} directly.
However, on the analogy of the fact that $f^\rho \in C^1$ when $f \in C^{\lceil \rho \rceil}$,
$D^s (|u|^{\rho-1} u)$ is expected to be controlled similarly when $s < \rho$.
Here $\lceil \cdot \rceil$ is the ceiling function
defined by $\lceil s \rceil = \min\{m \in \mathbb Z, m \geq s\}$.
In order to see this natural expectation,
one may need to extend Proposition \ref{Proposition:1.1} with $s \in (1,2)$.

We generalize this expectation slightly.
With $\rho > 1$,
we put $F_\rho \in C^{1}(\mathbb C)$ satisfying $F(0) = F'(0)=0$ and
	\begin{align}
	|F_\rho(u) - F_\rho(v)|
	&\lesssim
	\max(|u|,|v)^{\rho-1} | u - v|,
	\nonumber\\
	|F_\rho'(u) - F_\rho'(v)|
	&\lesssim
	\begin{cases}
	\max(|u|,|v)^{\rho-2} | u - v|
	&\mathrm{if} \quad \rho \geq 2,\\
	| u - v|^{\rho -1}
	&\mathrm{if} \quad \rho < 2,
	\end{cases}
	\label{eq:1.2}
	\end{align}
where
$a \lesssim b$ stands for $a \leq C b$ with some positive constant $C$.
We also denote $a \sim b$ when $a \lesssim b$ and $b \lesssim a$.
We use these notation though this paper.

In \cite{GOV94},
Ginibre, Ozawa, and Velo showed the expectation above in the framework of homogeneous Besov spaces
$\dot B_{p,q}^s$.
\begin{Proposition}[{\cite[Lemma:3.4]{GOV94}}]
\label{Proposition:1.2}
Let $d \geq 1$, $\rho \geq 1$, and $s \in (0,\min(2,\rho))$.
Let $1 \leq p, r \leq \infty$ and $(\rho-1)^{-1} \leq q \leq \infty$ satisfy \eqref{eq:1.1}.
If $u \in \dot B_{r,2}^s \cap L^{(\rho-1)q}$, then
	\begin{align}
	\| F_\rho(u) \|_{\dot B_{p,2}^s}
	\lesssim \| u \|_{L^{(\rho-1)q}}^{\rho-1} \| u \|_{\dot B_{r,2}^s}.
	\label{eq:1.3}
	\end{align}
\end{Proposition}
Since Besov spaces may play a role in useful auxiliary spaces,
Proposition \ref{Proposition:1.2} has been used
to study even $\dot H_p^s$-valued solutions to some nonlinear PDEs.
The advantage to consider the fractional chain rule in the framework of Besov space
is the following representation of homogeneous Besov norms:
	\[
	\| u \|_{\dot B_{r,q}^s}
	\sim
	\bigg(
	\int_0^\infty \lambda^{-s q-1} \sup_{|y| < \lambda} \|(\tau_y - 2 + \tau_{-y}) u \|_{L^r}^q d \lambda,
	\bigg)^{1/q},
	\]
where $\tau_y u = u(\cdot + y)$, $s \in (0,2)$, and $1 \leq p, q \leq \infty$
(See \cite[6.2.5. Theorem]{BL}, for example.).
With this representation,
the term $(\tau_y - 2 + \tau_{-y}) u$ gives a clear explanation
of the connection between the classical and fractional chain rules.
We remark that
even though the fractional differential operators are defined by Fourier multipliers,
it is nontrivial why $F'(u) \sim |u|^{\rho-1}$ appears as an upper bound of \eqref{eq:1.3}
from the view point Fourier multipliers.
We further remark that when $\rho =2$,
the chain rule may be shown by the argument of Fourier multipliers.
See \cite{FGO18}.

Proposition \ref{Proposition:1.2} is an effective estimate but
it seems more handy to close the argument only with Sobolev spaces.
So one of the main purpose of this paper is to show a similar estimate to Proposition \ref{Proposition:1.2}
in the framework of $\dot H_p^s$.
This is one of the main statement of this paper.
\begin{Proposition}
\label{Proposition:1.3}
Let $\rho > 1$ and $s \in (1,\rho)$.
Let $1 < p, r < \infty$ and $(\rho-1)^{-1} < q < \infty$ satisfy \eqref{eq:1.1}.
The estimate
	\[
	\| F_\rho(u) \|_{\dot H_p^s}
	\lesssim \| u \|_{L^{(\rho-1) q}}^{\rho-1} \| u \|_{\dot H_r^s}
	\]
holds for any $u \in \dot H_r^s \cap L^{(\rho-1)q}$.
\end{Proposition}

We remark that
the case when $p$ or $r=1, \infty$ is included in the assumption of Proposition \ref{Proposition:1.2}
but not in that of Proposition \ref{Proposition:1.3}.
It is because our proof is based on the boundedness of Hardy-Littlewood maximal operator.
See next section.

The main idea for the proof of Proposition \ref{Proposition:1.3}
is to express $\| F(u) \|_{\dot H_p^s}$ with $(\tau_y - 2 + \tau_{-y}) F(u)$
like the proof of Proposition \ref{Proposition:1.2}.
Since $\dot H_p^s$ seems not to be expressed like \eqref{eq:1.3},
we rewrite each dyadic components of $F(u)$ by $(\tau_y - 2 + \tau_{-y}) F(u)$.
Namely, the identity
	\begin{align}
	Q_j F(u)(x) = \frac 1 2 \int (\tau_y - 2 + \tau_{-y}) F(u)(x) \psi_j(y) dy
	\label{eq:1.4}
	\end{align}
plays a critical role in this paper,
where $Q_j$ is the standard $j$-th Littlewood-Paley dyadic operator
and $\psi_j$ is the corresponding kernel.
The details of $Q_j$ and $\psi_j$ are stated in next section.
We note that the identity
	\[
	Q_j F(u)(x) = \int (\tau_y - 1) F(u)(x) \psi_j(y) dy
	\]
plays an essential role in \cite{CW91} as well.
In this paper,
we deploy a similar but more careful approach with \eqref{eq:1.4}.

Our second purpose is to extend Corollary \ref{Corollary:1.4} below
following from Proposition \ref{Proposition:1.1}.
Here we put $a_+ = \max(a,0)$ and $a_- = \min(a,0)$ for $a \in \mathbb R$.
\begin{Corollary}
\label{Corollary:1.4}
Let $s \in (0,1)$, $\rho \geq 1$.
Let $1 \leq p < \infty$ and $1 < q, r < \infty$ satisfy
\eqref{eq:1.1}.
Then the estimate
	\begin{align*}
	&\| F_\rho (u) - F_\rho (v) \|_{\dot H^s_p}\\
	&\lesssim (\| u \|_{L^q(\rho-1)} + \| v\|_{L^q(\rho-1)} )^{\rho-1} \| u - v\|_{\dot H_r^s}\\
	&+ (\| u \|_{\dot H_r^s} + \| v\|_{\dot H_r^s} )
	(\| u \|_{L^q(\rho-1)} + \| u \|_{L^q(\rho-1)})^{(\rho-2)_+}
	\| u - v\|_{L^q(\rho-1)}^{\min(\rho-1,1)}
	\end{align*}
holds for any $u \in \dot H_r^s \cap L^{q(\rho-1)}$.
\end{Corollary}

Corollary \ref{Corollary:1.4} corresponds to the identity
	\[
	\nabla (u^\rho - v^\rho)
	= \rho u^{\rho-1} \nabla(u-v) + \rho ( u^{\rho-1} - v^{\rho-1} ) \nabla v .
	\]
Corollary \ref{Corollary:1.4} is as worthy as Proposition \ref{Proposition:1.1}
to show the locally-in-time well-posedness of semilinear PDEs.
We remark that Corollary \ref{Corollary:1.4} is not necessarily
to construct solutions with a contraction argument
because one may deploy a contraction argument in the ball of $\dot H_p^s$
with the distance of $L^p$, for example.
However, the proof of the continuous dependence of solution maps for initial data
may require Corollary \ref{Corollary:1.4}.

Our second purpose is to extend Corollary \ref{Corollary:1.4}
to the case where $s \in (1,2)$.
Namely, the second main statement of this paper is
\begin{Proposition}
\label{Proposition:1.5}
Let $\rho \in (1,2)$ and $s \in (1,\rho)$.
Let $1 \leq p < \infty$, $1 < r < \infty$, $(\rho-1)^{-1} < q < \infty$ satisfy
\eqref{eq:1.1}.
Let $d_{\rho,s} \in (0, \rho - s)$.
Then
	\begin{align*}
	&\| F_\rho (u) - F_\rho (v) \|_{\dot H^s_p}\\
	&\lesssim (\| u \|_{L^{q(\rho-1)}} + \| v\|_{L^{q(\rho-1)}} )^{\rho-1} \| u - v\|_{\dot H_r^s}\\
	&+ (\| u \|_{\dot H_r^s} + \| v\|_{\dot H_r^s} )
	(\| u \|_{L^{q(\rho-1)}} + \| v\|_{L^{q(\rho-1)}} )^{\rho - 1 - d_{\rho,s}}
	\| u - v\|_{L^{q(\rho-1)}}^{d_{\rho,s}}.
	\end{align*}
\end{Proposition}

We note that
an extension of Corollary \ref{Corollary:1.4} in the framework of Besov spaces was given in
\cite[Proposition 2.1]{ON99} and \cite[Lemma 6.2]{FIW}.
We remark that
on the analogy of classical chain rules,
it is expected that Proposition \ref{Proposition:1.5} may hold with $d_{\rho,s} = \rho -s$
but a similar technical difficulty arises in both Sobolev and Besov frameworks.
We also remark that the case where $\rho = s$ is hopeless to extend Proposition \ref{Proposition:1.5}
because of the failure of classical analogy, say $|x| \not \in C^1$, for example.

In the next section,
we collect some notation and basic estimates.
In section \ref{section:3},
we revisit the proof of Proposition \ref{Proposition:1.1} for the completeness.
In section \ref{section:4},
we give the proofs of Propositions \ref{Proposition:1.3} and \ref{Proposition:1.5}

\section{Preliminary}
\subsection{Notation}
Here we collect some notation.

Let $\mathcal S$ be the Schwartz class.
Let $\psi \in \mathcal S$ be radial and satisfy
	\[
	\mathrm{supp} \ \mathfrak F \psi \subset \{ \xi \ \mid \ 1/2 \leq |\xi| \leq 2 \}
	\]
and
	\[
	\sum_{j} \mathfrak F \psi(2^{-j} \xi) = 1
	\]
for $\xi \neq 0$.
For $j \in \mathbb Z$,
let $\psi_j = 2^{jn} \psi(2^{j} \cdot)$ so that $\| \psi_j \|_{L^1} = \| \psi \|_{L^1}$
and let $Q_j = \psi_j \ast$.
We also put $\widetilde Q_j = Q_{j-1} + Q_j + Q_{j+1}$ and $\widetilde \psi_j = \psi_{j-1} + \psi_j + \psi_{j+1}$.
We remark that $\widetilde Q_j Q_j = Q_j$ holds for any $j$.
It is known that
for $s \in \mathbb R$ and $1 < p,q < \infty$,
the homogeneous Sobolev and Triebel-Lizorkin norms are equivalent,
that is,
	\[
	\| f \|_{\dot H_p^s}
	\sim
	\| 2^{js} Q_j f \|_{L^p(\ell^2)}.
	\]
For the details of this equivalence,
we refer the reader to \cite[Theorem 5.1.2]{G}, for example.
For $f \in L_{\mathrm{loc}}^1$, we define the Hardy-Littlewood maximal operator by
	\[
	Mf(x) = \sup_{r > 0} \frac{1}{|B(r)|} \int_{B(r)} |f(x+y)| dy
	\]
and
	\[
	M^{(\eta)} f(x) = M(|f|^\eta)^{1/\eta}
	\]
for $\eta > 0$,
where $B(r) \subset \mathbb R^d$ is the ball with radius $r$ centered at the origin.
It is known that $M$ is bounded operator on $L^p$ and $L^p(\ell^q)$ for $1 < p,q < \infty$
(See \cite[Theorems 2.1.6 and 4.6.6]{G} and reference therein).
Moreover, for $1 \leq p < \infty$,
we define weighted $L^p$ norm with weight function $w$ by
	\[
	\| f \|_{L_w^p} = \bigg( \int |f|^p w dx \bigg)^{1/p}.
	\]

\subsection{Basic Estimates}
Here we collect some estimates.

\begin{Lemma}[{\cite[Theorem :2.1.10]{G}}]
\label{Lemma:2.1}
Let $w \in L^1$ be a positive radially decreasing function.
Then the estimate
	\[
	|w \ast g(x)| \leq \| w \|_{L^1} Mg(x)
	\]
holds for any $x \in \mathbb R^n$.
\end{Lemma}

\begin{Lemma}
\label{Lemma:2.2}
	\begin{align*}
	| \widetilde Q_j g(x+y) - \widetilde Q_j g(x) |
	&\lesssim
	M g(x) + M g(x+y),\\
	| \widetilde Q_j g(x+y) - 2 \widetilde Q_j g(x) + \widetilde Q_j g(x-y) |
	&\lesssim M g(x+y) + Mg(x) + M g(x-y).
	\end{align*}
Moreover, if $|y| < 2^{-j}$, then we have
	\begin{align*}
	| \widetilde Q_j g(x+y) - \widetilde Q_j g(x) |
	&\lesssim
	2^j |y| Mg(x),\\
	| \widetilde Q_j g(x+y) - 2 \widetilde Q_j g(x) + \widetilde Q_j g(x-y) |
	&\lesssim
	2^{2j} |y|^2 Mg(x).
	\end{align*}

\end{Lemma}

\begin{proof}
The estimates for $| \widetilde Q_j g(x+y) - \widetilde Q_j g(x) |$ is given in \cite{CW91}
but for completeness, we give the proof.

The first two estimate follows directly from Lemma \ref{Lemma:2.1}.

By the fundamental theorem calculus implies that the following identities hold:
	\[
	\psi_j(x+y) - \psi_j(x)
	= \int_0^1 (\nabla \psi)_j(x + \theta y) d \theta \cdot 2^j y
	\]
and
	\begin{align*}
	&\psi_j(x+y) - 2 \psi_j(x) + \psi_j(x-y)\\
	&= \int_0^1 ((\nabla \psi)_j(x+\theta y) - (\nabla \psi)_j(x-\theta y)) d \theta \cdot 2^j y\\
	&= \sum_{|\alpha|=2} \int_0^1 \int_0^1 (\partial^\alpha \psi)_j(x+ (2\theta'-1) \theta y)) d \theta'
	d \theta \cdot 2^{2j} y^2.
	\end{align*}
Since for $|\alpha| \leq 2$, if $|y| < 2^{-j}$,
	\[
	(\partial^\alpha \psi)_j(x+y)
	\lesssim 2^j (2+2^j|x+y|)^{-n-1}
	\leq 2^j (1+2^j|x|)^{-n-1},
	\]
Then the third and fourth estimates are obtained by
combining this and Lemma \ref{Lemma:2.1}.
\end{proof}

\begin{Lemma}
\label{Lemma:2.3}
For $0 < S < S'$, $P \geq 1$, and $a \in \ell^P$,
the following estimate holds:
	\[
	\| 2^{jS} \sum_{k} 2^{(k-j)_-S'} a_k \|_{\ell_j^P}
	\leq \bigg( \frac{1}{2^S-1} + \frac{1}{2^{S'-S}-1} \bigg) \| 2^{ks} a_k \|_{\ell_k^P}.
	\]
\end{Lemma}

\begin{proof}
By the Minkowski inequality, we have
	\begin{align*}
	\| 2^{jS} \sum_{k < j} 2^{(k-j)_-S'} a_k \|_{\ell_j^P}
	&= \| 2^{jS} \sum_{\ell < 0} 2^{\ell s'} a_{\ell + j} \|_{\ell_j^P}
	+ \| 2^{jS} \sum_{\ell \geq 0} a_{\ell + j} \|_{\ell_j^P}\\
	&\leq \sum_{\ell < 0} 2^{\ell (S'-S)} \| 2^{s(\ell+j)} a_{\ell + j} \|_{\ell_j^P}
	+ \sum_{\ell < 0} 2^{- \ell S} \| 2^{S(\ell+j)} a_{\ell + j} \|_{\ell_j^P}\\
	&\leq \frac{2}{2^{S'-S}-1} \| 2^{Sk} a_{k} \|_{\ell_k^P}
	+ \frac{1}{2^S-1} \| 2^{S k} a_{k} \|_{\ell_k^P}.
	\end{align*}
\end{proof}

\begin{Lemma}
\label{Lemma:2.4}
Let $0 < S < 1$ and $1 \leq P,Q < \infty$.
The following estimate holds:
	\begin{align*}
	&\| 2^{jS} \| v(y+x) (u(y + x) - u(x)) \|_{L_{\psi_j,y}^P} \|_{\ell_j^Q}\\
	& \lesssim M^{(P)} v(x) \| 2^{ks} M Q_k u  \|_{\ell_k^Q}
	+ \| 2^{ks} M^{(P)} (v M Q_k u)(x)  \|_{\ell_k^Q}
	\end{align*}
\end{Lemma}

\begin{proof}
Let $w \in L^1$ be smooth radially decreasing function satisfying
	\[
	w(x) \geq (1+|x|)^P |\psi(x)|
	\]
and put $w_j = 2^{jn} w(2^{j} \cdot)$.
Lemmas \ref{Lemma:2.1} and \ref{Lemma:2.2} imply that
we have
	\begin{align*}
	&\| v(y+x) (u(y + x) - u(x)) \|_{L_{\psi_j,y}^P}\\
	&= \sum_k \| v(y+x)(\widetilde Q_k Q_k u(y + x) - \widetilde Q_k Q_k u(x)) \|_{L_{\psi_j,y}^P}\\
	&\leq \sum_{k < j} 2^{k} \| y v(y+x) \|_{L_{\psi_j,y}^P} M Q_k u(x)\\
	&+ \sum_{k \geq j} \bigg(
		\| v(y+x) MQ_ku(x+y) \|_{L_{\psi_j,y}^P} + M Q_k u(x) \| v(y+x) \|_{L_{\psi_j,y}^P}\bigg)\\
	&\leq \| w \|_{L^1} \sum_{k < j} 2^{k-j} M^{(P)} v(x) M Q_k u(x)\\
	&+ \| w \|_{L^1} \sum_{k \geq j} \bigg( M^{(P)}( v MQ_ku) (x) + M^{(p)} v(x) M Q_k u(x) \bigg).
	\end{align*}
Therefore, Lemma \ref{Lemma:2.4} follows from these estimates above and Lemma \ref{Lemma:2.3}.
\end{proof}

\begin{Corollary}
\label{Corollary:2.5}
Let $0 < S < 1$, $1 \leq P_0 < 2$.
Let $P_0 < P <\infty$ and $1 < Q,R < \infty$ satisfy
	\[
	\frac 1 P = \frac 1 Q + \frac 1 R.
	\]
The estimates
	\[
	\| 2^{jS} \| (u(y + \cdot) - u(\cdot)) \|_{L_{\psi_j,y}^{P_0}} \|_{L^P(\ell_j^2)}
	\lesssim \| u \|_{\dot H_P^s}
	\]
and
	\[
	\| 2^{jS} \| v(y+\cdot) (u(y + \cdot) - u(\cdot)) \|_{L_{\psi_j,y}^{P_0}} \|_{L^P(\ell_j^2)}
	\lesssim \| v \|_{L^Q} \| u  \|_{\dot H_R^S}
	\]
hold.
\end{Corollary}

\begin{proof}
For $P_1 > P_0$,
the boundedness of $M$ implies that we have
	\[
	\| M^{(P_0)} f \|_{L^{P_1}}
	= \| M|f|^{P_0} \|_{L^{P_1/P_0}}^{1/P_0}
	\lesssim \| |f|^{P_0} \|_{L^{P_1/P_0}}^{1/P_0}
	= \| f \|_{L^{P_1}}.
	\]
Since $P_0 < P < Q$,
Corollary \ref{Corollary:2.5} follows from Lemma \ref{Lemma:2.4} and the H\"older inequality.
\end{proof}

\section{Revisit of the Chain rules when $s \in (0,1)$}
\label{section:3}
In this section,
we revisit the proof of Proposition \ref{Proposition:1.1} and Corollary \ref{Corollary:1.4}.
The proofs are essentially given in \cite{CW91} but for completeness,
we give the proofs.

\begin{proof}[Proof of Proposition \ref{Proposition:1.1}]
Since $\int \psi(y) dy = 0$,
the identity
	\[
	Q_j F(u)(x) = \int F(u)(x+y) - F(u)(y) \psi_j(y) dy
	\]
holds for any $x \in \mathbb R^n$ and $j \in \mathbb Z$.
Therefore, Proposition \ref{Proposition:1.1} follows from
Corollary \ref{Corollary:2.5} with $(v,P_0,P,Q,R,S)$ replaced by $(G(u),1,p,q,r,s)$.
\end{proof}

\begin{proof}[Proof of Corollary \ref{Corollary:1.4}]
The identity
	\begin{align*}
	&Q_j F_\rho(u)(x) - Q_j F_\rho(v)(x)\\
	&= \int (F_\rho(u)(x+y) - F_\rho(u)(y) + F_\rho(v)(x+y) - F_\rho(v)(x)) \psi_j(y) dy
	\end{align*}
holds.
The fundamental theorem of calculus implies that the identity 
	\begin{align*}
	&F_\rho(u)(x+y) - F_\rho(u)(y) + F_\rho(v)(x+y) - F_\rho(v)(x)\\
	&= \{ (u - v)(x+y) - (u-v)(x) \} \int_0^1 F_\rho'(\theta u(x+y) + (1-\theta) u(x)) d \theta\\
	&+ \{ v(x+y) - v(x) \} \\
	&\cdot \int_0^1 \{ F_\rho'(\theta u(x+y) + (1-\theta) u(x)) - F_\rho'(\theta v(x+y) + (1-\theta) v(x)) \}
	d \theta.
	\end{align*}
holds.
The identity above and \eqref{eq:1.2} imply that the estimate
	\begin{align*}
	&| F_\rho(u)(x+y) - F_\rho(u)(y) + F_\rho(v)(x+y) - F_\rho(v)(x)|\\
	&\lesssim ( |u(x+y)|^{\rho-1} + |u(x)|^{\rho-1}) | (u - v)(x+y) - (u-v)(x) |\\
	&+ ( |u(x+y)| + |u(x)| + |v(x+y)| + |v(x)|)^{(\rho-2)_+}\\
	&\cdot (|(u-v)(x+y)| + |(u-v)(x)|)^{\rho-1}
	| (u - v)(x+y) - (u-v)(x) |
	\end{align*}
holds.
Therefore, Corollary \ref{Corollary:1.4} follows from
Corollary \ref{Corollary:2.5} and the estimates above.
\end{proof}

\section{Proofs of the Chain rules when $s \in (1,2)$}
\label{section:4}
\begin{proof}[Proof of Proposition \ref{Proposition:1.3}]
We note that
it is shown in \cite[(3.26)]{GOV94} that the estimate
	\begin{align}
	|(\tau_y - 2 + \tau_{-y}) F_\rho(u)(x)|
	&\lesssim |u(x)|^{\rho-1} |u(x+y) - 2 u(x) + u(x-y)|
	\nonumber\\
	&+ |u(x+y) - u(x)|^\rho + |u(x-y) - u(x)|^\rho
	\label{eq:4.1}
	\end{align}
holds.
Moreover, the symmetry of $\psi$ and $\int \psi = 0$ imply that we have
	\begin{align}
	Q_j F_\rho(u)(x)
	&= \int F_\rho(u)(x+y) \psi_j(y) dy
	\nonumber\\
	&= \int F_\rho(u)(x-y) \psi_j(y) dy
	\nonumber\\
	&= \frac 1 2 \int (F_\rho(u)(x+y) + F_\rho(u)(x-y)) \psi_j(y) dy
	\nonumber\\
	&= \frac 1 2 \int (\tau_y - 2 + \tau_{-y}) F_\rho(u)(x) \psi_j(y) dy
	\label{eq:4.2}
	\end{align}
From \eqref{eq:4.1} and \eqref{eq:4.2},
the estimate
	\begin{align}
	|Q_j F(u)(x)|
	&\lesssim |u(x)|^{\rho - 1} \int |u(x+y) - 2 u(x) + u(x-y)| | \psi_j(y)| dy
	\nonumber\\
	&+ \int |u(x+y) - u(x)|^p |\psi_j(y)| dy
	\label{eq:4.3}
	\end{align}
follows.
Put $w (x)$ be radially decreasing function with $w(x) \geq (1+|x|^2) |\psi(x)|$.
By Lemma \ref{Lemma:2.2}, when $k<j$,
the first term on the RHS of \eqref{eq:4.3} is estimated by
	\begin{align*}
	&\int_{|y| < C 2^{-k}}
	| \widetilde Q_k Q_k u(x+y) - 2 \widetilde Q_k Q_k u(x) + \widetilde Q_k Q_k u(x-y) |
	|\psi_j(y)| dy\\
	&\leq M Q_k u(x) \int_{|y| < C 2^{-k}} 2^{2(k-j)} \omega_j(y) dy\\
	&\lesssim 2^{2(k-j)} M Q_k u(x).
	\end{align*}
Similarly, when $k \geq j$, it is estimated by
	\begin{align*}
	&\int_{|y| > 2^{-k}}
	| \widetilde Q_k Q_k u(x+y) - 2 \widetilde Q_k Q_k u(x) + \widetilde Q_k Q_k u(x-y) |
	|\psi_j(y) |dy\\
	&\lesssim 2^{2(k-j)} \int_{|y| > 2^{-k}}
	(2 M Q_k u(x+y) + 2 M Q_k u(x) |
	\omega_j(y) dy\\
	&\lesssim 2^{2(k-j)} (M Q_k u(x) + M^2 Q_k u(x)).
	\end{align*}
Therefore, the estimates above and Lemma \ref{Lemma:2.3} imply that we have
	\begin{align}
	&\| 2^{js} |u|^{\rho-1} \int |u(\cdot+y) - 2 u(\cdot) + u(\cdot-y)| | \psi_j(y)| dy \|_{L^p(\ell^2)}
	\nonumber\\
	&\lesssim \| 2^{ks} |u|^{\rho-1} (M Q_k u + M^2 Q_k u) \|_{L^p(\ell^2)}
	\nonumber\\
	&\lesssim \| u \|_{L^{q(\rho-1)}}^{\rho-1} \| u \|_{\dot H_r^s}.
	\label{eq:4.4}
	\end{align}
Moreover, by the Lemma \ref{Lemma:2.4},
the second term of the RHS of \eqref{eq:4.3} satisfy the estimate
	\begin{align}
	\| 2^{js} \|u(\cdot + y) - u(\cdot) \|_{L_{\psi_j}^\rho}^\rho \|_{L^p(\ell^2)}
	&= \| 2^{js/\rho} \|u(\cdot + y) - u(\cdot) \|_{L_{\psi_j}^\rho}\|_{L^{\rho p}(\ell^2)}^{\rho}
	\nonumber\\
	&\lesssim \| 2^{ks/\rho} M^{(\rho)} Q_k u \|_{L^{\rho p}(\ell^2)}^{\rho}
	\nonumber\\
	&\lesssim \| u \|_{\dot H_{\rho p}^{s/\rho}}^{\rho}
	\nonumber\\
	&\lesssim \| u \|_{L^{q(\rho-1)}}^{\rho-1} \| u \|_{\dot H_r^s},
	\label{eq:4.5}
	\end{align}
where we have used the Gagliardo-Nirenberg inequality to obtain the last estimate.
For the details of the Gagliardo-Nirenberg inequality,
we refer the reader to \cite[Corollary 2.4]{HMOW11} and reference therein.
Combining \eqref{eq:4.4} and \eqref{eq:4.5},
we obtain Proposition \ref{Proposition:1.3}.
\end{proof}

\begin{proof}[Proof of Proposition \ref{Proposition:1.5}]
For simplicity, we denote $d_{\rho,s}$ by $d$.
We note that the estimate
	\begin{align}
	&| (\tau_y -2 + \tau_{-y})(F_\rho (u) - F_\rho (v))(x)|
	\nonumber\\
	&\lesssim \max_{z \in \{x, x+y, x-y\} } |u(z)|^{\rho-1} | (\tau_y -2 + \tau_{-y}) (z-w)(x)|
	\nonumber\\
	& + \max_{z \in \{x, x+y, x-y\}} |u(z)|^{\rho-1} | (\tau_y -2 + \tau_{-y}) w(x)|
	\nonumber\\
	& + (|(\tau_y-1)u(x)| + |(\tau_{-y}-1)u(x)|)^{\rho-1} |(\tau_y-1)(u-v)(x)|
	\nonumber\\
	& + |(\tau_{-y}-1)u(x)| \min( \sigma, \max_{z \in \{x, x+y, x-y\}} |(u-v)(z)| )^{\rho-1}
	\label{eq:4.6}
	\end{align}
holds, where
	\[
	\sigma
	= |(\tau_y-1)u(x)| + |(\tau_{-y}-1)u(x)|
	+|(\tau_y-1)v(x)| + |(\tau_{-y}-1)v(x)|.
	\]
For the details of the estimate above, see \cite[Lemma:A.1]{FIW}.
Therefore,
the proof of proposition \ref{Proposition:1.5},
it is sufficient to show
	\begin{align}
	&\| 2^{js} \| B(u,v)(\cdot,y) \|_{L_{\psi_j,y}^1} \|_{L^p(\ell^2)}
	\nonumber\\
	&\lesssim
	(\| u \|_{\dot H_r^s} + \| v\|_{\dot H_r^s} )
	(\| u \|_{L^{q(\rho-1)}} + \| v\|_{L^{q(\rho-1)}} )^{\rho - 1 - d_{\rho,s}}
	\| u - v\|_{L^{q(\rho-1)}}^{d_{\rho,s}}.
	\label{eq:4.7}
	\end{align}
where
	\[
	B(u,v)(x,y)
	= | u(x+y) - u(x)| |v(x+y)-v(x)|^{\rho- 1 - d}
	|u(x+y)-v(x+y)|^d.
	\]
The other terms on the RHS of \eqref{eq:4.6}
may be treated like the argument above so we omit the detail.
Put
	\[
	q_0 = q \frac{\rho-1}{d}
	\qquad \mathrm{and} \qquad
	r_0 = \frac{p q_0}{q_0-p}.
	\]
Then the H\"older inequality and Lemma \ref{Lemma:2.1} imply that we have
	\begin{align*}
	&\| B(u,v)(x,y) \|_{L^1_{\psi_j,y}}\\
	&\leq \| |(u-v)(x+y)|^d \|_{L^{q_0/p}_{\psi_j,y}}
	\| |u(x+y)-u(x)| | v(x+y) - v(x)|^{\rho-1 - d} \|_{L^{r_0/p}_{\psi_j,y}}\\
	&\lesssim (M^{(q (\rho-1)/p)}|u-v|(x))^d\\
	&\cdot \| u(x+y)-u(x) \|_{L_{\psi_j,y}^{(\rho-d)r_0/p}}
	\| v(x+y)-v(x) \|_{L_{\psi_j,y}^{(\rho-d)r_0/p}}^{\rho- d-1}.
	\end{align*}
Then Lemma \ref{Lemma:2.1} and Corollary \ref{Corollary:2.5} imply that
the estimates
	\begin{align*}
	&\| 2^{js} \| B(u,v)(\cdot,y) \|_{L_{\psi_j,y}^1} \|_{L^p(\ell^2)}\\
	&\lesssim \| M^{(q (\rho-1)/p)}|u-v| \|_{L^{dq_0}}^d\\
	&\cdot \| 2^{j s/(\rho-d)}
	\| u(\cdot+y)-u(\cdot) \|_{L_{\psi_j,y}^{(\rho-d)r_0/p}} \|_{L^{r_0(\rho-d)}(\ell^2)}\\
	&\cdot \| 2^{j s/(\rho-d)} \| v(\cdot+y)-v(\cdot) \|_{L_{\psi_j,y}^{(\rho-d)r_0/p}}
	\|_{L^{r_0(\rho-d)}(\ell^2)}^{\rho - d - 1}\\
	&\lesssim \| u-v \|_{L^{q(\rho-1)}}^d
	\| u \|_{\dot H_{r_0(\rho-d)}^{s/(\rho-d)}}
	\| v \|_{\dot H_{r_0(\rho-d)}^{s/(\rho-d)}}^{\rho-d-1}
	\end{align*}
hold.
Then \eqref{eq:4.7} follows from the estimates above and the Gagliardo-Nirenberg inequality.
\end{proof}

\begin{Remark}
In the proof above,
one cannot take $d = \rho - s$
because
	\[
	\| 2^{j} \| u(\cdot+y)-u(\cdot) \|_{L_{\psi_j,y}^{(\rho-d)r_0/p}} \|_{L^{r_0(\rho-d)}(\ell^2)}
	\]
is not controlled by Lemma \ref{Lemma:2.4}.
\end{Remark}

\section*{Acknowledgment}
The author is supported by JSPS Early-Career Scientists 20K14337.



\end{document}